\documentclass[a4paper,12pt]{amsart}
\usepackage{amsmath,latexsym,amssymb,amsfonts}
\usepackage{amscd,amssymb,epsfig}
\usepackage[arrow, matrix, curve]{xy}
\usepackage{hyperref}

\numberwithin{equation}{section}

\newtheorem{theorem}{Theorem}[section]
\newtheorem{lemma}[theorem]{Lemma}
\newtheorem{corollary}[theorem]{Corollary}
\newtheorem{proposition}[theorem]{Proposition}

\theoremstyle{definition}
\newtheorem{definition}[theorem]{Definition}

\newtheorem{example}[theorem]{Example}
\newtheorem{remark}[theorem]{Remark}

\renewcommand{\epsilon}{\varepsilon}

\newcommand{\A}{\mathcal{A}}
\newcommand{\C}{\mathbb{C}}
\newcommand{\E}{\mathbb{E}}
\newcommand{\F}{\mathbb{F}}
\newcommand{\M}{\mathcal{M}}
\newcommand{\N}{\mathbb{N}}
\newcommand{\R}{\mathbb{R}}

\newcommand{\id}{\operatorname{id}}
\newcommand{\ev}{\operatorname{ev}}

\newcommand{\ran}{\operatorname{ran}}
\newcommand{\vN}{\operatorname{vN}}

\makeindex

\begin{document}

\title[Absence of algebraic relations and of zero divisors]{Absence of algebraic relations and of zero divisors under the assumption of full non-microstates free entropy dimension}

\author[T. Mai]{Tobias Mai}
\address{Universit\"{a}t des Saarlandes, FR $6.1-$Mathematik, 66123 Saarbr\"{u}cken, Germany}
\email{mai@math.uni-sb.de}

\author[R. Speicher]{Roland Speicher}
\address{Universit\"{a}t des Saarlandes, FR $6.1-$Mathematik, 66123 Saarbr\"{u}cken, Germany}
\email{speicher@math.uni-sb.de}

\author[M. Weber]{Moritz Weber}
\address{Universit\"{a}t des Saarlandes, FR $6.1-$Mathematik, 66123 Saarbr\"{u}cken, Germany}
\email{weber@math.uni-sb.de}

\date{\today}

\thanks{This work was supported by the ERC Advanced Grant "Non-commutative distributions in free probability" and by a grant from the DFG
(SP-419-8/1).
\\
We thank H. Bercovici, Y. Dabrowski and D. Shlyakhtenko for discussions and comments on an earlier version of this paper.}



\begin{abstract}
We show that in a tracial and finitely generated $W^\ast$-probability space existence of conjugate variables excludes algebraic relations for the generators.
Moreover, under the assumption of maximal non-microstates free entropy dimension, we prove that there are no zero divisors in the sense that the product of any non-commutative polynomial in the generators with any element from the von Neumann algebra is zero if and only if at least one of those factors is zero. 
In particular, this shows that in this case the distribution of any non-constant self-adjoint non-commutative polynomial in the generators does not have atoms.

Questions on the absence of atoms for polynomials in\linebreak non-commuting random variables (or for polynomials in random matrices) have been an open problem for quite a while. We solve this general problem by showing that maximality of free entropy dimension excludes atoms.
\end{abstract}

\maketitle

\section{Introduction}

In a groundbreaking series of papers \cite{Voi-Entropy-I, Voi-Entropy-II, Voi-Entropy-III, Voi-Entropy-IV, Voi-Entropy-V, Voi-Entropy-VI} (see also the survey \cite{Voi-Entropy-Surv}), Voiculescu transferred the notion of entropy and Fisher information to the world of\linebreak non-commutative probability theory.
Free entropy and free Fisher information are some of the core quantities in free probability theory, with
fundamental importance both for operator algebraic and random matrix questions. One of the most striking results which came out of this program is arguably the proof of the fact \cite{Voi-Entropy-III} that the free group factors do not possess Cartan subalgebras. This gave in particular the solution of the by then longstanding open question, whether every separable $\text{II}_1$-factor contains Cartan subalgebras. 
But despite such deep results and applications, still many of the basic analytic properties of free entropy and Fisher information are poorly understood.

Voiculescu gave actually two different approaches to entropy and Fisher information in the non-commutative setting. The first one is based on the notion of matricial microstates and defines free entropy $\chi$ first and then, based on this, the free Fisher information $\Phi$; the second approach is based on the notion of conjugate variables with respect to certain non-commutative derivatives and defines free Fisher information $\Phi^\ast$ first and then, based on this, free entropy $\chi^\ast$. We want to note that both constructions lead independently to objects $\chi$ and $\chi^*$ (as well as $\Phi$ and $\Phi^*$)  which are, in analogy with the classical theory, justifiably called entropy (and Fisher information), but it is still not known whether they coincide. For many questions the actual value of these quantities is not important, essential is whether they are finite or infinite. There exist also more refined quantities, so-called free entropy dimensions (again in various variations), which allow a further distinction of the case of infinite entropy. In particular, finiteness of free entropy or of free Fisher information implies that the microstates free entropy dimension $\delta^*$ takes on its maximal value.

In the classical case, finiteness of entropy or of Fisher information imply some regularity of the corresponding distribution of the variables; in particular, they have a density (with respect to Lebesgue measure). In the non-commutative situation, the notion of a density does not make any direct sense, but still it is believed that the existence of finite free entropy or finite free Fisher information (in any of the two approaches) should correspond to some regularity properties of the considered non-commutative distributions. Thus one expects many ''smooth" properties for random variables $X_1,\dots,X_n$ for which either one of the quantities $\chi(X_1,\dots,X_n)$, $\chi^*(X_1,\dots,X_n)$, $\Phi(X_1,\dots,X_n)$, or $\Phi^*(X_1,\dots,X_n)$ is finite. In particular, it is commonly expected that such a finiteness implies that
\begin{itemize}
\item
there cannot exist non-trivial algebraic relations between the considered random variables;
\item
and that such algebraic relations can also not hold locally on non-trivial Hilbert subspaces; more formally this
means that there are no zero divisors in the affiliated von Neumann algebra. 
\end{itemize}
Up to now there has been no proof of such general statements. We will show here such results. In the first preprint version of this paper this was done under the assumption of finite non-microstates free Fisher information $\Phi^*(X_1\dots,X_n)$. Inspired by this, Shlyakhtenko could prove in \cite{Shly2014}, by combining our ideas with his earlier work in \cite{Connes-Shlyakhtenko}, our results under the weaker assumption of maximal
non-microstates free entropy dimension $\delta^\ast(X_1,\dots,X_n)=n$.
In the light of this, we reexamined our original arguments and noticed that they can also be extended without much extra effort to this most general case.

Our work was originally inspired by the realization that in the usual approaches to conjugate variables one usually assumes that there exist no algebraic relations between the considered variables. Though this is not necessary for the definition of conjugate variables themselves, more advanced arguments (which rely on the existence of non-commutative derivative operators) only work in the absence of such algebraic relations. As alluded to above one actually expected that the existence of conjugate variables (and thus the finiteness of $\Phi^*$) implies the absence of such relations. But since this has not been shown up to now, there was a kind of an annoying gap in the theory. This gap will be closed in Section \ref{algebraic_relations}.

It turned out that our ideas for this could also be extended to the much deeper question whether relations could hold locally; instead of asking whether for a non-trivial polynomial $P$ we can have algebraically $P(X_1,\dots,X_n)=0$,
we weaken this to the question whether $P(X_1,\dots,X_n)$ could be zero on an affiliated Hilbert subspace; if $u$ denotes the projection onto this subspace, then this is the question whether we can have a zero divisor $u$
for $P(X_1,\dots,X_n)$, i.e., an element $u$ in the von Neumann algebra generated by $X_1,\dots,X_n$ such that $P(X_1,\dots,X_n)u=0$.
We will show that already the condition $\delta^\ast(X_1,\dots,X_n)=n$ for the non-microstates free entropy dimension excludes such zero divisors.

In particular, this result allows to conclude that the distribution of any non-trivial self-adjoint non-commutative polynomial in the generators does not have atoms, if the generators have full non-microstates free entropy dimension. Note that in a random matrix language this allows to conclude that the asymptotic eigenvalue distribution of polynomials in random matrices has, under the above assumption, no atomic parts. Questions on the absence of atoms for polynomials in non-commuting random variables (or for polynomials in random matrices) have been an open
problem for quite a while. Only recently there was some progress on this in such generality; in \cite{Shlyakthenko-Skoufranis},
Shlyakhtenko and Skoufranis showed that polynomials in free variables exhibit (under the assumption of no atoms for
each of the variables) no atoms; here we give a vast generalization of this, by showing that not freeness is the crucial issue but maximality of free entropy dimension.

\section{Existence of conjugate variables and absence of algebraic relations}\label{algebraic_relations}

Let $\C\langle Z_1,\dots,Z_n\rangle$ be the $\ast$-algebra of non-commutative polynomials in $n$ self-adjoint (formal) variables $Z_1,\dots,Z_n$. For $j=1,\dots,n$, we denote by $\partial_j$ the non-commutative derivative with respect to $Z_j$, i.e. $\partial_j$ is the unique derivation
$$\partial_j:\ \C\langle Z_1,\dots,Z_n\rangle \to \C\langle Z_1,\dots,Z_n\rangle \otimes \C\langle Z_1,\dots,Z_n\rangle$$
that satisfies $\partial_j Z_i = \delta_{i,j} 1 \otimes 1$ for $i=1,\dots,n$.

Recall that a linear map $\delta: \A \rightarrow \M$, which is defined on a complex unital algebra $\A$ and takes its values in a $\A$-bimodule $\M$, is called a derivation if it satisfies the Leibniz rule
$$\delta(a_1 a_2) = \delta(a_1) \cdot a_2 + a_1 \cdot \delta(a_2) \qquad\text{for all $a_1,a_2\in\A$}.$$

In particular, being a derivation means for $\partial_j$ that
\begin{equation}\label{derivation-def}
\partial_j (P_1 P_2) = (\partial_j P_1) (1 \otimes P_2) + (P_1 \otimes 1) (\partial_j P_2)
\end{equation}
holds for all $P_1,P_2\in \C\langle Z_1,\dots,Z_n\rangle$. According to the $\C\langle Z_1,\dots,Z_n\rangle$-bimodule structure of $\C\langle Z_1,\dots,Z_n\rangle \otimes \C\langle Z_1,\dots,Z_n\rangle$, we will often abbreviate \eqref{derivation-def} to
$$\partial_j (P_1 P_2) = (\partial_j P_1) P_2 + P_1 (\partial_j P_2).$$

More explicitly, $\partial_j$ acts on monomials $P$ as
$$\partial_j P = \sum_{P = P_1 Z_j P_2} P_1 \otimes P_2,$$
where the sum runs over all decompositions $P=P_1Z_jP_2$ of $P$ with monomials $P_1,P_2$.

We record here the easy but useful observation that each derivation $\delta: \C\langle Z_1,\dots,Z_n\rangle \rightarrow \M$ taking values in any $\C\langle Z_1,\dots,Z_n\rangle$-bimodule $\M$ is determined by its values $\delta(Z_1),\dots,\delta(Z_n)$, which can be expressed explicitly by
\begin{equation}\label{derivation-uniqueness}
\delta(P) = \sum^n_{j=1} \partial_j(P) \sharp \delta(Z_j) \qquad\text{for all $P\in\C\langle Z_1,\dots,Z_n\rangle$}.
\end{equation}
Here, we stipulate that, whenever $\M$ is an $\A$-bimodule, the symbol $\sharp$ denotes the mapping $\sharp: \A\otimes\A \times \M \to \M$, which is determined by bilinear extension of $(a_1 \otimes a_2)\sharp m := a_1 \cdot m \cdot a_2$.

Throughout the following, let $(M,\tau)$ be a tracial $W^\ast$-probability space, which means that $M$ is a von Neumann algebra and $\tau$ is a faithful normal tracial state on $M$. For selfadjoint $X_1,\dots,X_n\in M$ we 
denote by $\vN(X_1,\dots,X_n)\subset M$ the von Neumann subalgebra of $M$ which is generated by $X_1,\dots,X_n$ and by
$L^2(X_1,\dots,X_n,\tau)\subset L^2(M,\tau)$ the $L^2$-space which is generated by $X_1,\dots,X_n$ with respect to the inner product given by $\langle P,Q\rangle:=\tau(PQ^*)$.

In the following, we will denote by $\ev_X$, for a given $n$-tuple $X=(X_1,\dots,X_n)$ of self-adjoint elements of $M$, the homomorphism
$$\ev_X:\ \C\langle Z_1,\dots,Z_n\rangle \rightarrow \C\langle X_1,\dots,X_n\rangle \subset M$$
given by evaluation at $X=(X_1,\dots,X_n)$, i.e. the homomorphism $\ev_X$ is determined by $Z_i \mapsto X_i$.

For reasons of clarity, we put $P(X) := \ev_X(P)$ for any $P\in \C\langle Z_1,\dots,Z_n\rangle$ and $Q(X) := (\ev_X \otimes \ev_X)(Q)$ for any $Q\in\C\langle Z_1,\dots,Z_n\rangle^{\otimes 2}$.

\begin{definition}\label{conjugate_system}
Let $X_1,\dots,X_n\in M$ be self-adjoint elements. If there are elements $\xi_1,\dots,\xi_n\in L^2(M,\tau)$, such that
\begin{equation}\label{conjugate_relations}
(\tau\otimes\tau)((\partial_jP)(X_1,\dots,X_n)) = \tau(\xi_j P(X_1,\dots,X_n))
\end{equation}
is satisfied for each non-commutative polynomial $P\in\C\langle Z_1,\dots,Z_n\rangle$ and for $j=1,\dots,n$, then we say that $(\xi_1,\dots,\xi_n)$ \emph{satisfies the conjugate relations for $(X_1,\dots,X_n)$}.

If, in addition, $\xi_1,\dots,\xi_n$ belong to $L^2(X_1,\dots,X_n,\tau)$, we say that $(\xi_1,\dots,\xi_n)$ is the \emph{conjugate system for $(X_1,\dots,X_n)$}. 
\end{definition}

Like in the usual setting, we have the following.

\begin{remark}
Let $\pi$ be the orthogonal projection from $L^2(M,\tau)$ to $L^2(X_1,\dots,X_n,\tau)$. It is easy to see that if $(\xi_1,\dots,\xi_n)$ satisfies the conjugate relations for $(X_1,\dots,X_n)$, then $(\pi(\xi_1),\dots,\pi(\xi_n))$ satisfies the conjugate relations for $(X_1,\dots,X_n)$ as well, and is therefore a conjugate system for $(X_1,\dots,X_n)$.

It is an easy consequence of its defining property \eqref{conjugate_relations} that a conjugate system $(\xi_1,\dots,\xi_n)$ for $(X_1,\dots,X_n)$ is unique if it exists. 
\end{remark}

Note that our notion of conjugate relations and conjugate variables differs from the usual definition which was given by Voiculescu in \cite{Voi-Entropy-V}, roughly speaking, just by the placement of brackets. To be more precise, in \eqref{conjugate_relations}, we first apply the derivative $\partial_j$ to the given non-commutative polynomial $P$ before we apply the evaluation at $X=(X_1,\dots,X_n)$, instead of applying the evaluation first, which consequently makes it necessary to have in the second step a well-defined derivation on $\C\langle X_1,\dots,X_n\rangle$ corresponding to $\partial_j$.

From a more abstract point of view, this idea is in the same spirit as \cite[Lemma 3.2]{Shly2000} but only on a purely algebraic level. In fact, we used the surjective homomorphism $\ev_X: \C\langle Z_1,\dots,Z_n\rangle \rightarrow \C\langle X_1,\dots,X_n\rangle$ in order to pass from $(\C\langle X_1,\dots,X_n\rangle,\tau)$ to the non-commutative probability space $(\C\langle Z_1,\dots,Z_n\rangle,\tau_X)$, where $\tau_X := \tau\circ\ev_X$. Due to this lifting, the algebraic relations between the generators disappear whereas the relevant information about their joint distribution remains unchanged.

In this section, our aim is to show that the existence of a conjugate system guarantees that $X_1,\dots,X_n$ do not satisfy any algebraic relations. This will be the content of Theorem \ref{main1} below. Its proof proceeds in two steps, which are performed in the following two propositions.
 
\begin{proposition}\label{key-prop-1}
Let $(M,\tau)$ be a tracial $W^\ast$-probability space and let self-adjoint $X_1,\dots,X_n\in M$ be given. Assume that there are elements $\xi_1,\dots,\xi_n\in L^2(M,\tau)$, such that $(\xi_1,\dots,\xi_n)$ satisfies the conjugate relations \eqref{conjugate_relations} for $X=(X_1,\dots,X_n)$.

Then the following implication holds true for any non-commutative polynomial $P\in \C\langle Z_1,\dots,Z_n\rangle$:
\begin{equation}\label{reduction}
P(X) = 0 \quad\Longrightarrow\quad \forall j=1,\dots,n:\ (\partial_j P)(X) = 0
\end{equation}
\end{proposition}

Before beginning with the proof, let us introduce a binary operation $\sharp$ on the algebraic tensor product $M \otimes M$ by bilinear extension of
$$(a_1 \otimes a_2) \sharp (b_1 \otimes b_2) := (a_1 b_1) \otimes (b_2 a_2).$$
Note that, since $M \otimes M$ is naturally a $M$-bimodule, this corresponds exactly to our earlier definition of $\sharp$ and will therefore not lead to any confusion.

\begin{proof}[Proof of Proposition \ref{key-prop-1}]
Assume that $P\in \C\langle Z_1,\dots,Z_n\rangle$ satisfies $P(X)=0$ for $X=(X_1,\dots,X_n)$ and choose any $j=1,\dots,n$. If we take arbitrary $P_1,P_2\in \C\langle Z_1,\dots,Z_n\rangle$, we have by iterating the product rule \eqref{derivation-def} twice that
$$\partial_j (P_1 P P_2) = (\partial_j P_1) P P_2 + P_1 P (\partial_j P_2) + P_1 (\partial_j P) P_2$$
and therefore, by evaluating this identity at $X$ and applying $\tau\otimes\tau$ subsequently,
$$(\tau \otimes \tau)\big((\partial_j (P_1 P P_2))(X)\big) = (\tau \otimes \tau)(P_1(X) (\partial_j P)(X) P_2(X)).$$
Furthermore, according to \eqref{conjugate_relations}, we may deduce that
$$(\tau \otimes \tau)\big((\partial_j (P_1 P P_2))(X)\big) = \tau(\xi_j (P_1 P P_2)(X)) = 0.$$
Thus, we observe that
$$(\tau \otimes \tau)((P_1 \otimes P_2)(X) \sharp (\partial_j P)(X)) = (\tau \otimes \tau)(P_1(X) (\partial_j P)(X) P_2(X)) = 0$$
for all $P_1,P_2 \in \C\langle Z_1,\dots,Z_n\rangle$ and hence by linearity
$$(\tau \otimes \tau)(Q(X) \sharp (\partial_j P)(X)) = 0$$
for all $Q\in \C\langle Z_1,\dots,Z_n\rangle^{\otimes 2}$. If we apply this observation to $Q = (\partial_j P)^\ast$, the faithfulness of $\tau\otimes\tau$ (recall that $\tau$ was assumed to be faithful) implies $(\partial_j P)(X) = 0$, as claimed.
\end{proof}

The second proposition shows now that the validity of \eqref{reduction} is already sufficient to exclude algebraic relations.

\begin{proposition}\label{key-prop-2}
Let $(M,\tau)$ be a tracial $W^\ast$-probability space and let self-adjoint $X_1,\dots,X_n\in M$ be given, such that \eqref{reduction} holds.

Then the non-commutative random variables $X_1,\dots,X_n$ do not satisfy any non-trivial algebraic relation, i.e., if $P(X_1,\dots,X_n) = 0$ holds for any non-commutative polynomial $P \in \C\langle Z_1,\dots,Z_n\rangle$, then we must have $P=0$.
\end{proposition}

\begin{proof}
For $j=1,\dots,n$, we may define $\Delta_j := ((\tau\circ\ev_X)\otimes\id) \circ \partial_j$, which gives a linear mapping
$$\Delta_j:\ \C\langle Z_1,\dots,Z_n\rangle \rightarrow \C\langle Z_1,\dots,Z_n\rangle.$$
The assumption \eqref{reduction} yields then that for any $P\in\C\langle Z_1,\dots,Z_n\rangle$ the implication
\begin{equation}\label{reduction-reduced}
P(X) = 0 \quad\Longrightarrow\quad \forall j=1,\dots,n:\ (\Delta_j P)(X) = 0.
\end{equation}
holds true with $X=(X_1,\dots,X_n)$.

Assume now that there is any non-zero polynomial $P\in\C\langle Z_1,\dots,Z_n\rangle$ such that $P(X_1,\dots,X_n)=0$ holds, for instance
$$P(Z_1,\dots,Z_n) = a_0 + \sum_{k=1}^d \sum_{i_1,\dots,i_k = 1}^n a_{i_1,\dots,i_k} Z_{i_1} \dots Z_{i_k},$$
where $d\geq1$ denotes the total degree of $P$. We choose any summand of highest degree
$$a_{i_1,\dots,i_d} Z_{i_1} \dots Z_{i_d}$$
of $P$. Since $\Delta_{i_d} \dots \Delta_{i_1}$ is clearly zero on constants, any monomial of degree strictly less than $d$, and furthermore on any monomial of degree $d$, where the variables do not appear in the prescribed order, we see that $\Delta_{i_d} \dots \Delta_{i_1} P = a_{i_1,\dots,i_d}$. Hence, we deduce by iterating \eqref{reduction-reduced}
$$a_{i_1,\dots,i_n} = (\Delta_{i_d} \dots \Delta_{i_1} P)(X) = 0,$$
which finally leads to a contradiction. Therefore, we must have $P=0$, which concludes the proof.
\end{proof}

Combining the above Proposition \ref{key-prop-1} with Proposition \ref{key-prop-2} leads us now directly to the following theorem.

\begin{theorem}\label{main1}
As before, let $(M,\tau)$ be a tracial $W^\ast$-probability space. Let $X_1,\dots,X_n\in M$ be self-adjoint and assume that there are elements $\xi_1,\dots,\xi_n\in L^2(M,\tau)$, such that $(\xi_1,\dots,\xi_n)$ satisfies the conjugate relations for $(X_1,\dots,X_n)$, i.e. \eqref{conjugate_relations} holds for $j=1,\dots,n$. Then we have the following statements:
\begin{itemize}
 \item[(a)] $X_1,\dots,X_n$ do not satisfy any non-trivial algebraic relation, i.e. there exists no non-zero polynomial $P\in\C\langle Z_1,\dots,Z_n\rangle$ such that $P(X_1,\dots,X_n)=0$.
 \item[(b)] For $j=1,\dots,n$, there is a unique derivation
$$\hat{\partial}_j:\ \C\langle X_1,\dots,X_n\rangle \rightarrow \C\langle X_1,\dots,X_n\rangle \otimes \C\langle X_1,\dots,X_n\rangle$$
which satisfies $\hat{\partial}_j(X_i) = \delta_{j,i} 1 \otimes 1$ for $i=1,\dots,n$.
\end{itemize}
\end{theorem}

Note that part (b) is an immediate consequence of part (a): Since (a) tells us that the evaluation homomorphism $\ev_X$ is in fact an isomorphism, we can immediately define a non-commutative derivation
$$\hat{\partial}_j:\ \C\langle X_1,\dots,X_n\rangle \rightarrow \C\langle X_1,\dots,X_n\rangle \otimes \C\langle X_1,\dots,X_n\rangle,$$
where the terminology derivation has to be understood with respect to the $\C\langle X_1,\dots,X_n\rangle$-bimodule structure of $\C\langle X_1,\dots,X_n\rangle^{\otimes 2}$. The uniqueness can be deduced from \eqref{derivation-uniqueness}; see also Proposition \ref{derivatives_relations} in the appendix.

Following \cite{Voi-Entropy-V}, we may proceed now by defining (non-microstates) free Fisher information.

\begin{definition}
Let $(M,\tau)$ be a tracial $W^\ast$-probability space and let self-adjoint elements $X_1,\dots,X_n \in M$ be given. We define their \emph{(non-microstates) free Fisher information $\Phi^\ast(X_1,\dots,X_n)$} by
$$\Phi^\ast(X_1,\dots,X_n) := \sum^n_{j=1} \|\xi_j\|_2^2$$
if a conjugate system $(\xi_1,\dots,\xi_n)$ for $(X_1,\dots,X_n)$ in the sense of Definition \ref{conjugate_system} exists, and we put $\Phi^\ast(X_1,\dots,X_n) := \infty$ if no such conjugate system for $(X_1,\dots,X_n)$ exists.
\end{definition}

This $\Phi^\ast(X_1,\dots,X_n)$ is just the usual non-microstates free Fisher information as defined in \cite{Voi-Entropy-V}. However, we have now the advantage that it can be defined even without assuming the algebraic freeness of $X_1,\dots,X_n$ right from the beginning. Actually, our result can now be stated as follows: $\Phi^*(X_1,\dots,X_n)<\infty$ implies the absence of algebraic relations between $X_1,\dots,X_n$.

Let $(M,\tau)$ be a $W^\ast$-probability space and let self-adjoint elements $X_1,\dots,X_n\in M$ be given such that the condition $\Phi^\ast(X_1,\dots,X_n) < \infty$ is fulfilled. Part (a) of Theorem \ref{main1} tells us then that $X_1,\dots,X_n$ do not satisfy any algebraic relation, which in other words means that the evaluation homomorphism $\ev_X$ induces an isomorphism between the abstract polynomial algebra $\C\langle Z_1,\dots, Z_n\rangle$ and the subalgebra $\C\langle X_1,\dots,X_n\rangle$ of $M$. Thus, each of the non-commutative derivatives
$$\hat{\partial}_j:\ \C\langle X_1,\dots,X_n\rangle \rightarrow \C\langle X_1,\dots,X_n\rangle \otimes \C\langle X_1,\dots,X_n\rangle,$$
whose existence is claimed in part (b) of Theorem \ref{main1}, is naturally induced under this identification by the corresponding non-commutative derivative $\partial_j$ on $\C\langle Z_1,\dots,Z_n\rangle$. Due to this strong relationship, we do not have to distinguish anymore between $\partial_j$ and $\hat{\partial}_j$.

We finish this section by noting that $\Phi^\ast(X_1,\dots,X_n)<\infty$ moreover excludes analytic relations. More precisely, this means that there is no non-zero non-commutative power series $P$, which is convergent on a polydisc
$$D_R:=\{(Y_1,\dots,Y_n)\in M^n|\ \forall j=1,\dots,n:\ \|Y_j\| < R\}$$
for some $R>0$, such that $(X_1,\dots,X_n)\in D_R$ and $P(X_1,\dots,X_n)=0$. Based on Voiculescu's original definition of the non-microstates free Fisher information and hence under the additional assumption that $X_1,\dots,X_n$ are algebraically free, this was shown by Dabrowski in \cite[Lemma 37]{Dab-free_stochastic_PDE}.

\section{Non-microstates free entropy dimension and zero divisors}

Inspired by the methods used in the proof of Theorem \ref{main1}, we address now the more general question of existence of zero divisors under the assumption of full non-microstates free entropy dimension.

First of all, we shall make more precise what we mean by this. We postpone the definition of the non-microstates free entropy dimension $\delta^\ast(X_1,\dots,X_n)$ and related quantities to Subsection \ref{subsec:definitions}, but we state here the result that we aim to prove.

\begin{theorem}\label{main2}
Let $(M,\tau)$ be a tracial $W^\ast$-probability space. Furthermore, let $X_1,\dots,X_n \in M$ be self-adjoint elements and assume that $\delta^\ast(X_1,\dots,X_n) = n$ holds.

Then, for any non-zero non-commutative polynomial $P$, there exists no non-zero element $w\in\vN(X_1,\dots,X_n)$ such that
$$P(X_1,\dots,X_n) w = 0.$$
\end{theorem}

Recall that to each element $X=X^\ast\in M$ corresponds a unique probability measure $\mu_X$ on the real line $\R$, which has the same moments as $X$, i.e. it satisfies 
$$\tau(X^k) = \int_\R t^k\, d\mu_X(t) \qquad\text{for $k=0,1,2,\dots$}.$$
It is an immediate consequence of Theorem \ref{main2} that the distribution $\mu_{P(X_1,\dots,X_n)}$ of $P(X_1,\dots,X_n)$ for any non-constant self-adjoint polynomial $P$ cannot have atoms, if $\delta^\ast(X_1,\dots,X_n)=n$. The precise statement reads as follows. 

\begin{corollary}\label{cor-main2}
Let $(M,\tau)$ be a tracial $W^\ast$-probability space and let $X_1,\dots,X_n \in M$ be self-adjoint with $\delta^\ast(X_1,\dots,X_n)=n$.

Then, for any non-constant self-adjoint non-commutative polynomial $P$, the distribution $\mu_{P(X_1,\dots,X_n)}$ of $P(X_1,\dots,X_n)$ does not have atoms.
\end{corollary}

Indeed, any atom $\alpha$ of the distribution $\mu_{P(X_1,\dots,X_n)}$, i.e. any $\alpha\in\R$ satisfying $\mu_{P(X_1,\dots,X_n)}(\{\alpha\})\neq 0$, leads by the spectral theorem for bounded self-adjoint operators on Hilbert spaces to a non-zero projection $u$ satisfying $(P(X_1,\dots,X_n) - \alpha 1)u = 0$. Thus, applying Theorem \ref{main2} yields immediately the statement of Corollary \ref{cor-main2}.

We point out that the conclusions of both Theorem \ref{main2} and Corollary \ref{cor-main2} were shown under the stronger assumption of finite non-microstates free Fisher information in an earlier version of this paper. In the above stated generality they appeared first in \cite{Shly2014}, where the proof is based on results from \cite{Connes-Shlyakhtenko}. We are going to prove those statements in a more direct way by refining our initial methods.

More precisely, in Subsection \ref{subsec:finite_Fisher}, we will give a quantitative version of our key idea that under the assumption of finite non-microstates Fisher information there is a strong relation between kernels of polynomials and the kernels of their derivatives.

Since the semicircular perturbation is inherent in the definition of the non-microstates free entropy as well as the corresponding entropy dimension, we find ourselves in the setting of finite free Fisher information. Therefore, we will study in Subsection \ref{subsec:entropy_dimension} the behavior of the results found in Subsection \ref{subsec:finite_Fisher} under semicircular perturbations -- roughly speaking, we will be interested in the case where the perturbation tends to zero.

Finally, in Subsection \ref{subsec:proof_thm}, which is dedicated to the proof of Theorem \ref{main2}, we will deduce from this observation a certain reduction argument that allows us to reduce successively the degree of the polynomial $P$ satisfying the conditions of Theorem \ref{main2}.

\subsection{Non-microstates free entropy and free entropy dimension}
\label{subsec:definitions}

We want to catch up now on the definition of the non-microstates free entropy dimension.

Let $(M,\tau)$ be a tracial $W^\ast$-probability space and let self-adjoint elements $X_1,\dots,X_n\in M$ be given.

By possibly enlarging $(M,\tau)$, we may always assume that $(M,\tau)$ contains additionally semicircular elements $S_1,\dots,S_n$ such that
$$\{X_1,\dots,X_n\}, \{S_1\}, \dots, \{S_n\}$$
are freely independent. Indeed, this can be done by replacing $(M,\tau)$ by the free product $(M,\tau) \ast_\C (L(\F_n),\tau_n)$ of $(M,\tau)$ with the free group factor $(L(\F_n),\tau_n)$.

Following Voiculescu \cite{Voi-Entropy-V}, we define the \emph{non-microstates free entropy} $\chi^\ast(X_1,\dots,X_n)$ of $X_1,\dots,X_n$ by
\begin{align*}
\lefteqn{\chi^\ast(X_1,\dots,X_n)}\\
&\quad := \frac{1}{2} \int^\infty_0\Big(\frac{n}{1+t}-\Phi^\ast(X_1+\sqrt{t}S_1,\dots,X_n+\sqrt{t}S_n)\Big)\, dt + \frac{n}{2}\log(2\pi e).
\end{align*}
We need to note that the function
$$t\mapsto \Phi^\ast(X_1+\sqrt{t}S_1,\dots,X_n+\sqrt{t}S_n)$$
is well-defined, since \cite[Corollary 3.9]{Voi-Entropy-V} tells us that there exists a conjugate system of $(X_1+\sqrt{t}S_n,\dots,X_n+\sqrt{t}S_n)$ for all $t>0$. Moreover, we have the inequalities (cf. \cite[Corollary 6.14]{Voi-Entropy-V})
\begin{equation}\label{eq:Fisher_perturbation}
\frac{n^2}{C^2 + nt} \leq \Phi^\ast(X_1+\sqrt{t}S_1,\dots,X_n+\sqrt{t}S_n) \leq \frac{n}{t} \quad\text{for all $t>0$},
\end{equation}
where $C^2 := \tau(X_1^2 + \dots + X_n^2)$. The left inequality in \eqref{eq:Fisher_perturbation} particularly implies that (cf. \cite[Proposition 7.2]{Voi-Entropy-V})
$$\chi^\ast(X_1,\dots,X_n) \leq \frac{n}{2}\log(2\pi e n^{-1} C^2).$$

This allows to define the \emph{non-microstates free entropy dimension} $\delta^\ast(X_1,\dots,X_n)$ by
$$\delta^\ast(X_1,\dots,X_n) := n - \liminf_{\epsilon \searrow 0} \frac{\chi^\ast(X_1+\sqrt{\epsilon}S_1,\dots,X_n+\sqrt{\epsilon}S_n)}{\log(\sqrt{\epsilon})}.$$

We note that there is actually a variant of $\delta^\ast(X_1,\dots,X_n)$, given by
$$\hat\delta^\star(X_1,\dots,X_n) := n - \liminf_{t\searrow0} t \Phi^\ast(X_1+\sqrt{t}S_1,\dots,X_n+\sqrt{t}S_n),$$
which is formally obtained by applying L'Hospital's rule to the $\liminf$ appearing in the definition of $\delta^\ast(X_1,\dots,X_n)$. In \cite{Connes-Shlyakhtenko}, where $\hat\delta^\star$ was introduced, it was denoted by
$\delta^\star$; we have slightly changed the notation for better legibility.
Due to \eqref{eq:Fisher_perturbation}, we have that $0 \leq \hat\delta^\star(X_1,\dots,X_n) \leq n$. 

It was already mentioned in \cite{Shly2014} that $\delta^\ast(X_1,\dots,X_n)=n$ or even $\hat\delta^\star(X_1,\dots,X_n)=n$ are the weakest possible assumptions where we can expect a version of Theorem \ref{main2} to hold true. Accordingly, it sits at the end of a longer chain of general implications, namely
\begin{align*}
\Phi^\ast(X_1,\dots,X_n) < \infty & \quad \Longrightarrow \quad \chi^\ast(X_1,\dots,X_n) > - \infty\\
                                  & \quad \Longrightarrow \quad \delta^\ast(X_1,\dots,X_n)=n\\
                                  & \quad \Longrightarrow \quad \hat\delta^\star(X_1,\dots,X_n)=n.
\end{align*}
The first implication follows by definition of $\chi^\ast(X_1,\dots,X_n)$, the second implication is a direct consequence of the definition of $\delta^\ast(X_1,\dots,X_n)$, and the last implication is justified by
$\hat\delta^\star(X_1,\dots,X_n) \geq \delta^\ast(X_1,\dots,X_n),$
which was shown in \cite[Lemma 4.1]{Connes-Shlyakhtenko} by a straightforward computation.

\subsection{The case of finite non-microstates free Fisher information revisited}
\label{subsec:finite_Fisher}

Let $(M,\tau)$ be a tracial $W^\ast$-probability space and let $X_1,\dots,X_n \in M$ be self-adjoint with $\Phi^\ast(X_1,\dots,X_n) < \infty$.

As we have seen in Section \ref{algebraic_relations}, those conditions guarantee that, for $j=1,\dots,n$, there exists a unique derivation
$$\partial_j:\ \C\langle X_1,\dots,X_n\rangle \rightarrow \C\langle X_1,\dots,X_n\rangle \otimes \C\langle X_1,\dots,X_n\rangle,$$
which is determined by the condition $\partial_j X_i = \delta_{i,j} 1 \otimes 1$ for $i=1,\dots,n$. For each $j=1,\dots,n$, we may consider $\partial_j$ as a densely defined unbounded linear operator
$$\partial_j:\ L^2(X_1,\dots,X_n,\tau) \supseteq D(\partial_j) \rightarrow L^2(X_1,\dots,X_n, \tau \otimes \tau)$$
with domain $D(\partial_j) := \C\langle X_1,\dots,X_n\rangle$. Note that as $L^2(X_1,\dots,X_n,\tau)$ stands for the closure of $\C\langle X_1,\dots,X_n\rangle$ in $L^2(M,\tau)$, the closure of $\C\langle X_1,\dots,X_n\rangle^{\otimes 2}$ in the $L^2$-space $L^2(M \overline{\otimes} M,\tau \otimes \tau)$, constructed over the von Neumann algebra tensor product $M \overline{\otimes} M$ of $M$ with itself, is denoted here by $L^2(X_1,\dots,X_n,\tau \otimes \tau)$.

Since due to $\Phi^\ast(X_1,\dots,X_n) < \infty$ a conjugate system $(\xi_1,\dots,\xi_n)$ for $(X_1,\dots,X_n)$ exists, we see by \eqref{conjugate_relations} that $1 \otimes 1$ belongs to the domain of definition of the adjoints $\partial_1^\ast,\dots,\partial_n^\ast$ and that $\xi_j = \partial_j^\ast(1 \otimes 1)$ holds for $j=1,\dots,n$.

The subsequent considerations will be based on several well-known facts about those operators $\partial_j$, which we collect here for reader's convenience.

\begin{lemma}[Corollary 4.2 and Proposition 4.3 in \cite{Voi-Entropy-V}]\label{Voi-formula}
Let $(M,\tau)$ be a tracial $W^\ast$-probability space. If $X_1,\dots,X_n \in M$ are self-adjoint with $\Phi^\ast(X_1,\dots,X_n) < \infty$, then we have for $j=1,\dots,n$ that
$$\C\langle X_1,\dots,X_n\rangle \otimes \C\langle X_1,\dots,X_n\rangle \subseteq D(\partial_j^\ast),$$
i.e. $\partial_j^\ast$ is densely defined as well and thus $\partial_j$ is closable. More explicitly, we have for each $Y\in \C\langle X_1,\dots,X_n\rangle^{\otimes 2}$ the formula
$$\partial_j^\ast(Y) = m_{\xi_j}(Y) - m_1(\id\otimes\tau\otimes\id)(\partial_j\otimes\id+\id\otimes\partial_j)(Y).$$
\end{lemma}

Here, for any $\eta\in L^2(M,\tau)$, we denote by $m_\eta$ the linear operator $m_\eta: M\otimes M \rightarrow L^2(M,\tau)$ defined on the algebraic tensor product $M\otimes M$, which is given by $m_\eta(a_1 \otimes a_2) := a_1\eta a_2$. And thus of course, $m_1(a_1\otimes a_2)=a_1a_2$.

The lemma above allows us to conclude that in particular
\begin{equation}\label{conjugate_formulas}
\begin{aligned}
\partial_j^\ast(P\otimes 1) &= P \xi_j - (\id\otimes\tau)(\partial_j P),\\
\partial_j^\ast(1\otimes P) &= \xi_j P - (\tau\otimes\id)(\partial_j P)
\end{aligned}
\end{equation}
holds for $j=1,\dots,n$ and any $P\in\C\langle X_1,\dots,X_n\rangle$.

\begin{lemma}[Lemma 12 in \cite{Dab-Gamma}]\label{Dab-estimates}
Let $(M,\tau)$ be a tracial $W^\ast$-probability space. If $X_1,\dots,X_n \in M$ are self-adjoint such that the condition $\Phi^\ast(X_1,\dots,X_n) < \infty$ is satisfied, then we have for each $P\in\C\langle X_1,\dots,X_n\rangle$ that
\begin{equation}\label{Dab-estimates-1}
\begin{aligned}
\|P \xi_j - (\id\otimes\tau)(\partial_j P)\|_2 &\leq \|\xi_j\|_2 \|P\|,\\
\|\xi_j P - (\tau\otimes\id)(\partial_j P)\|_2 &\leq \|\xi_j\|_2 \|P\|
\end{aligned}
\end{equation}
and therefore
\begin{equation}\label{Dab-estimates-2}
\begin{aligned}
\|(\id\otimes\tau)(\partial_j P)\|_2 &\leq 2\|\xi_j\|_2 \|P\|,\\
\|(\tau\otimes\id)(\partial_j P)\|_2 &\leq 2\|\xi_j\|_2 \|P\|.
\end{aligned}
\end{equation}
\end{lemma}

Note that Lemma \ref{Dab-estimates} is actually a slight extension of Lemma 12 in \cite{Dab-Gamma}, since we added in \eqref{Dab-estimates-1} and \eqref{Dab-estimates-2} each time the second estimates. In fact, they can be easily deduced from the first estimates by using the more general identity
\begin{equation}\label{duality}
(\tau\otimes\id)(P_1 (\partial_i P_2))^\ast = (\id\otimes\tau)((\partial_i P_2^\ast) P_1^\ast)
\end{equation}
for all $P_1,P_2\in\C\langle X_1,\dots,X_n\rangle$, which itself can easily be checked on monomials.

Moreover, we note that thanks to \eqref{conjugate_formulas}, the inequalities in \eqref{Dab-estimates-1} can be rewritten in the following way:
\begin{equation}\label{conjugate_norms}
\begin{aligned}
\|\partial_j^\ast(P \otimes 1)\|_2 &\leq \|\xi_j\|_2 \|P\|\\
\|\partial_j^\ast(1 \otimes P)\|_2 &\leq \|\xi_j\|_2 \|P\|
\end{aligned}
\end{equation}

We will need the following extension of \eqref{Dab-estimates-2} and of \eqref{conjugate_norms}.

\begin{corollary}\label{key-estimates}
In the situation described above, the following holds true:
\begin{itemize}
 \item[(a)] For all $Y_1,Y_2\in \C\langle X_1,\dots,X_n\rangle$, we have for $i=1,\dots,n$ that
$$\|\partial_i^\ast(Y_1 \otimes Y_2)\|_2 \leq 3 \|\xi_i\|_2 \|Y_1\| \|Y_2\|.$$
 \item[(b)] For all $Y_1,Y_2\in \C\langle X_1,\dots,X_n\rangle$, we have for $i=1,\dots,n$ that
\begin{align*}
\|(\id\otimes\tau)((\partial_i Y_1) 1 \otimes Y_2)\|_2 & \leq 4 \|\xi_i\|_2 \|Y_1\| \|Y_2\|,\\
\|(\tau\otimes\id)(Y_1 \otimes 1 (\partial_i Y_2))\|_2 & \leq 4 \|\xi_i\|_2 \|Y_1\| \|Y_2\|.
\end{align*}
\end{itemize}
\end{corollary}

\begin{proof}
According to Voiculescu's formula, which we recalled in Lemma \ref{Voi-formula}, we have for all $Y_1,Y_2 \in \C\langle X_1,\dots,X_n\rangle$ that
\begin{align*}
 \partial_i^\ast (Y_1 \otimes Y_2)
 &= Y_1 \xi_i Y_2 - m_1(\id\otimes\tau\otimes\id)(\partial_i \otimes \id + \id \otimes \partial_i) (Y_1 \otimes Y_2)\\
 &= Y_1 \xi_i Y_2 - (\id\otimes\tau)(\partial_i Y_1) Y_2 - Y_1 (\tau\otimes\id)(\partial_i Y_2)\\
 &= \partial_i^\ast(Y_1 \otimes 1) Y_2 - Y_1 (\tau\otimes\id)(\partial_i Y_2)\\
 &= \partial_i^\ast(Y_1 \otimes 1) Y_2 - Y_1 (\id\otimes\tau)(\partial_i Y_2^\ast)^\ast
\end{align*}
and thus, by applying the estimates \eqref{Dab-estimates-2} and \eqref{conjugate_norms}, that
\begin{align*}
\|\partial_i^\ast(Y_1 \otimes Y_2)\|_2
 &\leq \|\partial_i^\ast(Y_1 \otimes 1)\|_2 \|Y_2\| + \|Y_1\| \|(\id\otimes\tau)(\partial_i Y_2^\ast)\|_2\\
 &\leq 3 \|\xi_i\|_2 \|Y_1\| \|Y_2\|.
\end{align*}
This shows the validity of (a). For proving (b), we first follow the idea of ``integration by parts'' in order to obtain
\begin{align*}
(\id\otimes\tau)((\partial_i Y_1) 1 \otimes Y_2)
 &= (\id\otimes\tau)(\partial_i(Y_1 Y_2)) - (\id\otimes\tau)(Y_1 \otimes 1 (\partial_i Y_2))\\
 &= (\id\otimes\tau)(\partial_i(Y_1 Y_2)) - Y_1 (\id\otimes\tau)(\partial_i Y_2)
\end{align*}
for arbitrary $Y_1,Y_2\in\C\langle X_1,\dots,X_n\rangle$. From this, we can easily deduce by using \eqref{Dab-estimates-2} that
\begin{align*}
\|(\id\otimes\tau)(&(\partial_i Y_1) 1 \otimes Y_2)\|_2\\
 &\leq \|(\id\otimes\tau)(\partial_i(Y_1 Y_2))\|_2 + \|Y_1\| \|(\id\otimes\tau)(\partial_i Y_2)\|_2\\
 &\leq 4 \|\xi_i\|_2 \|Y_1\| \|Y_2\|
\end{align*}
which is the first inequality. The second inequality can be proven similarly.
\end{proof}

We will also need the following easy application of Kaplansky's density theorem.

\begin{lemma}\label{approximation}
For any $w=w^\ast\in\vN(X_1,\dots,X_n)$, there exists a sequence $(w_k)_{k\in\N}$ of elements in $\C\langle X_1,\dots,X_n\rangle$ such that:
\begin{itemize}
 \item[(i)] $w_k = w_k^\ast$
 \item[(ii)] $\displaystyle{\sup_{k\in\N} \|w_k\| \leq \|w\|}$
 \item[(iii)] $\|w_k-w\|_2 \rightarrow 0$ for $k\rightarrow\infty$
\end{itemize}
\end{lemma}

\begin{proof}
First of all, we note that in order to prove the statement, it suffices to prove existence of a sequence $(w_k)_{k\in\N}$ of elements in $\C\langle X_1,\dots,X_n\rangle$, which satisfy only conditions (ii) and (iii). Indeed, if we replace in this case $w_k$ by its real part $\Re(w_k)=\frac{1}{2}(w_k+w_k^\ast)$, conditions (ii) and (iii) are still valid, but we have achieved condition (i) in addition.

For proving existence under these weaker conditions, we may apply Kaplansky's density theorem. This guarantees the existence of a sequence $(w_k)_{k\in\N}$ of elements in $\C\langle X_1,\dots,X_n\rangle$, such that $\|w_k\| \leq \|w\|$ for all $k\in\N$, which particularly implies (ii), and which converges to $w$ in the strong operator topology. It remains to note that, with respect to the weak operator topology, $w_k^\ast w \rightarrow w^\ast w$, $w^\ast w_k \rightarrow w^\ast w$, and $w_k^\ast w_k \rightarrow w^\ast w$ as $k\rightarrow\infty$, such that according to the continuity of $\tau$
\begin{align*}
\|w_k-w\|_2^2 &= \tau((w_k-w)^\ast(w_k-w))\\
              &= \tau(w_k^\ast w_k) - \tau(w_k^\ast w) - \tau(w^\ast w_k) + \tau(w^\ast w)
\end{align*}
tends to $0$ as $k\rightarrow 0$, which shows (iii) and thus concludes the proof.
\end{proof}

The main result of this subsection is the following proposition. There, we will use the projective norm $\|\cdot\|_\pi$ on $\C\langle X_1,\dots,X_n\rangle^{\otimes 2}$, which is given by
$$\|Q\|_\pi := \inf\bigg\{\sum^N_{k=1} \|Q_{k,1}\| \|Q_{k,2}\|\bigg|\ Q = \sum^N_{k=1} Q_{k,1} \otimes Q_{k,2}\bigg\}$$
for any $Q\in \C\langle X_1,\dots,X_n\rangle^{\otimes 2}$, where the infimum is taken over all possible decompositions of $Q$ with $Q_{k,1},Q_{k,2}\in\C\langle X_1,\dots,X_n\rangle$ for $k=1,\dots,N$ and some $N\in\N$.

Whenever it becomes necessary to mention explicitly the dependence of $\|\cdot\|_\pi$ on the underlying set of generators $X=(X_1,\dots,X_n)$, we will also write $\|\cdot\|_{\pi,X}$ instead of $\|\cdot\|_\pi$.

\begin{proposition}\label{Fisher-bound}
Let $(M,\tau)$ be a tracial $W^\ast$-probability space and let self-adjoint elements $X_1,\dots,X_n\in M$ be given such that the condition $\Phi^\ast(X_1,\dots,X_n) < \infty$ is satisfied. Let $(\xi_1,\dots,\xi_n)$ be the conjugate system for $(X_1,\dots,X_n)$. Then, for all $P\in\C\langle X_1,\dots,X_n \rangle$ and all $u,v\in\vN(X_1,\dots,X_n)$, we have
\begin{equation}\label{eq:Fisher-bound}
|\langle v^\ast (\partial_i P) u, Q\rangle| \leq 4 \|\xi_i\|_2 \bigl(\|Pu\|_2 \|v\| + \|u\| \|P^\ast v\|_2\bigr) \|Q\|_\pi
\end{equation}
for all $Q\in \C\langle X_1,\dots,X_n\rangle^{\otimes 2}$ and $i=1,\dots,n$. In particular, we have
\begin{equation}
\begin{aligned}
\lefteqn{\sum^n_{i=1} |\langle v^\ast (\partial_i P) u, Q\rangle|^2}\\
 & \qquad \leq 16 \bigl(\|Pu\|_2 \|v\| + \|u\| \|P^\ast v\|_2\bigr)^2 \Phi^\ast(X_1,\dots,X_n) \|Q\|_\pi^2
\end{aligned}
\end{equation}
for all $Q\in \C\langle X_1,\dots,X_n\rangle^{\otimes 2}$.
\end{proposition}

\begin{proof}
Firstly, let us consider $u,v\in \C\langle X_1,\dots,X_n\rangle$. In this particular case, we observe that for all $Q_1,Q_2\in\C\langle X_1,\dots,X_n\rangle$
\begin{align*}
\langle Pu, & \partial_i^\ast(vQ_1 \otimes Q_2)\rangle\\
&= \langle \partial_i (Pu), vQ_1 \otimes Q_2\rangle\\
&= \langle (\partial_i P) u, vQ_1 \otimes Q_2\rangle + \langle P (\partial_i u), vQ_1 \otimes Q_2\rangle\\
&= \langle v^\ast (\partial_i P) u, Q_1 \otimes Q_2\rangle + \langle (\partial_i u) Q_2^\ast , P^*vQ_1 \otimes 1\rangle\\
&= \langle v^\ast (\partial_i P) u, Q_1 \otimes Q_2\rangle + \langle (\id\otimes\tau)((\partial_i u) Q_2^\ast) , P^*vQ_1\rangle
\end{align*}
holds. Rearranging the above equation yields, by using Corollary \ref{key-estimates},
\begin{align*}
|\langle v^\ast &(\partial_i P) u, Q_1 \otimes Q_2\rangle|\\
&\leq |\langle Pu, \partial_i^\ast(vQ_1 \otimes Q_2)\rangle| + |\langle (\id\otimes\tau)((\partial_i u) Q_2^\ast) , P^\ast vQ_1\rangle|\\
&\leq \|Pu\|_2 \|\partial_i^\ast(vQ_1 \otimes Q_2)\|_2 + \|(\id\otimes\tau)((\partial_i u) Q^*_2)\|_2 \|P^\ast vQ_1\|_2\\
&\leq 3 \|\xi_i\|_2 \|Pu\|_2 \|v\| \|Q_1\| \|Q_2\| + 4 \|\xi_i\|_2 \|u\| \|P^\ast v\|_2 \|Q_1\|\|Q^*_2\|\\
&\leq 4 \|\xi_i\|_2 \bigl(\|Pu\|_2 \|v\| + \|u\| \|P^\ast v\|_2\bigr) \|Q_1\| \|Q_2\|.
\end{align*}
Due to Lemma \ref{approximation}, the obtained inequality
$$|\langle v^\ast (\partial_i P) u, Q_1 \otimes Q_2\rangle|\leq 4 \|\xi_i\|_2 
\bigl(\|Pu\|_2 \|v\| + \|u\| \|P^\ast v\|_2\bigr) \|Q_1\| \|Q_2\|$$
also holds for all $u,v\in\vN(X_1,\dots,X_n)$. Moreover, it follows then for all $Q\in \C\langle X_1,\dots,X_n\rangle^{\otimes 2}$ that
$$|\langle v^\ast (\partial_i P) u, Q \rangle|\leq 4 \|\xi_i\|_2 \bigl(\|Pu\|_2 \|v\| + \|u\| \|P^\ast v\|_2\bigr) \|Q\|_\pi.$$
This shows the first part of the statement. The second inequality follows by taking squares on both sides and summing over $i=1,\dots,n$. Since $\Phi^\ast(X_1,\dots,X_n) = \sum^n_{i=1}\|\xi_i\|^2_2$, this concludes the proof.
\end{proof}

We want to stress that Proposition \ref{Fisher-bound} is in fact a quantitative version of a previous result of ours that allowed us to give in an earlier version of the present paper a proof of Theorem \ref{main2} under the weaker assumption of finite non-microstates Fisher information. This was based on the following corollary.

\begin{corollary}
Let $(M,\tau)$ be a tracial $W^\ast$-probability space and let self-adjoint elements $X_1,\dots,X_n \in M$ be given, such that
$$\Phi^\ast(X_1,\dots,X_n) <\infty$$
holds. We consider $P\in\C\langle Z_1,\dots,Z_n\rangle$. Then, for arbitrary $u,v\in\vN(X_1,\dots,X_n)$, the following implication holds true:
$$P(X) u = 0 \quad\text{and}\quad P(X)^\ast v = 0 \implies \forall i=1,\dots,n:\ v^\ast (\partial_i P)(X) u  = 0,$$
where we abbreviate $X=(X_1,\dots,X_n)$.
\end{corollary}

\begin{proof}
The inequality \eqref{eq:Fisher-bound}, which was stated in Proposition \ref{Fisher-bound}, immediately implies that $\langle v^\ast (\partial_i P) u, Q\rangle = 0$ for all $Q\in \C\langle X_1,\dots,X_n\rangle^{\otimes 2}$. Hence, since $\C\langle X_1,\dots,X_n\rangle^{\otimes 2}$ is dense in $L^2(M,\tau)$ with respect to $\|\cdot\|_2$, this yields $v^\ast (\partial_i P)(X) u  = 0$ as claimed. 
\end{proof}

Thus, readers interested in the proof of Theorem \ref{main2} only under the stronger assumption $\Phi^\ast(X_1,\dots,X_n)<\infty$ may skip Subsection \ref{subsec:entropy_dimension} and proceed directly to Subsection \ref{subsec:proof_thm}, since the final step in the proof of Theorem \ref{main2} will only need the above reduction argument.

\subsection{Treating the case of full entropy dimension via semicircular perturbations}
\label{subsec:entropy_dimension}

Since the non-microstates free entropy dimension $\delta^\ast(X_1,\dots,X_n)$ and its variant $\hat\delta^\star(X_1,\dots,X_n)$ are both determined in a more or less explicit way by the behavior of the function
$$t \mapsto \Phi^\ast(X_1+\sqrt{t} S_1,\dots, X_n+\sqrt{t} S_n)$$
as $t\searrow0$, one is tempted to apply the results obtained in Proposition \ref{Fisher-bound} to the semicircular perturbation $(X_1+\sqrt{t} S_1,\dots, X_n+\sqrt{t} S_n)$. In fact, in this way the quantity
\begin{align*}
\alpha(X_1,\dots,X_n) & := n - \hat{\delta}^\star(X_1,\dots,X_n)\\
                      & = \liminf_{t\searrow0} t \Phi^\ast(X_1+\sqrt{t}S_1,\dots,X_n+\sqrt{t}S_n),
\end{align*}
which also appeared in \cite[Section 4]{Connes-Shlyakhtenko}, emerges naturally from the inequality given in Proposition \ref{Fisher-bound} and allows us to study its influence.

It is therefore not surprising that some of the technical arguments that we will use below for dealing with semicircular perturbations are similar to \cite[Section 4]{Connes-Shlyakhtenko}. However, the proof itself is conceptually independent and follows a different strategy, since it is a straightforward continuation of Proposition \ref{Fisher-bound} and hence relies on the inequalities due to Dabrowski \cite{Dab-Gamma}, which we recalled in Lemma \ref{Dab-estimates}.

More precisely, we will show the following.

\begin{proposition}\label{Fisher-growth}
Let $(M,\tau)$ be a tracial $W^\ast$-probability space and let self-adjoint elements $X_1,\dots,X_n \in M$ be given. Moreover, let $P\in\C\langle Z_1,\dots,Z_n\rangle$ be a non-commutative polynomial for which there are elements $u,v\in\vN(X_1,\dots,X_n)$ such that
$$P(X_1,\dots,X_n)u=0 \qquad\text{and}\qquad P(X_1,\dots,X_n)^\ast v=0.$$
Then, for all $Q\in\C\langle X_1,\dots,X_n\rangle^{\otimes 2}$,
\begin{align*}
\lefteqn{\sum^n_{i=1} |\langle v^\ast (\partial_i P)(X) u, Q\rangle|^2}\\
& \qquad \leq 16 \big(\|(\partial P)(X)u\|_2 \|v\| + \|u\| \|(\partial P)(X)^\ast v\|_2\big)^2 \alpha(X) \|Q\|_\pi^2,
\end{align*}
where we abbreviate $X=(X_1,\dots,X_n)$.
\end{proposition}

Here, we use the notation $\partial P$ for the gradient $(\partial_1 P, \dots, \partial_n P)$ of a polynomial $P\in\C\langle Z_1,\dots,Z_n\rangle$. Evaluation $(\partial P)(X)$, taking adjoints $(\partial P)(X)^\ast$, and multiplication by elements from $M$ like in $(\partial P)(X)u$ and $(\partial P)(X)^\ast v$ are then defined component-wise.

Furthermore, we point out that the space $L^2(M,\tau)^n$ becomes a Hilbert space in the obvious way. We denote its induced norm also by $\|\cdot\|_2$.

\begin{proof}[Proof of Proposition \ref{Fisher-growth}]
(i) Without any restriction, we may assume that our underlying $W^\ast$-probability space $(M,\tau)$ contains $n$ normalized semicircular elements $S_1,\dots,S_n$ such that
$$\{X_1,\dots,X_n\},\{S_1\},\dots,\{S_n\}$$
are freely independent. We define variables
$$X_j^t := X_j + \sqrt{t} S_j \qquad\text{for $t\geq0$ and $j=1,\dots,n$}$$
and denote by $N_t := \vN(X_1^t,\dots,X_n^t)$ the von Neumann algebras they generate. In particular, $N_0$ is the von Neumann algebra generated by $X_1,\dots,X_n$. We abbreviate $X^t=(X_1^t,\dots,X_n^t)$ for $t\geq 0$, so that in particular $X^0=X=(X_1,\dots,X_n)$.

Since $N_t$, for each $t\geq0$, is a von Neumann subalgebra of $M$, we may consider the trace-preserving conditional expectation $\E_t$ from $M$ onto $N_t$. Finally, we introduce $u_t := \E_t[u] \in N_t$ and $v_t := \E_t[v]\in N_t$.

It follows then that $P(X^t) u_t = P(X^t) \E_t[u] = \E_t[P(X^t)u]$ and hence
$$\|P(X^t) u_t\|_2 = \|\E_t[P(X^t)u]\|_2 \leq \|P(X^t)u\|_2.$$
Now, since $t\mapsto P(X^t)u$ is a polynomial in $\sqrt{t}$, which vanishes at $t=0$, we may observe that
$$\lim_{t\searrow 0} \frac{1}{\sqrt{t}} P(X^t)u = \sum^n_{i=1} (\partial_i P)(X)u \sharp S_i.$$
Since the linear subspaces
$$\operatorname{span}\{a_1 S_j a_2|\ a_1,a_2 \in N_0\},\qquad j=1,\dots,n,$$
of $L^2(M,\tau)$ are pairwise orthogonal and since the mapping
$$L^2(N_0,\tau) \otimes L^2(N_0,\tau) \to L^2(M,\tau),\ U \mapsto U \sharp S_j$$
is in fact an isometry, which are both consequences of the assumed freeness of $\{X_1,\dots,X_n\},\{S_1\},\dots,\{S_n\}$, it follows that
$$\lim_{t\searrow 0} \frac{1}{\sqrt{t}} \|P(X^t)u\|_2 = \bigg(\sum^n_{i=1} \|(\partial_i P)(X)u\|_2^2\bigg)^{1/2} = \|(\partial P)(X)u\|_2$$
Similarly, we may deduce that
$$\lim_{t\searrow 0} \frac{1}{\sqrt{t}} \|P(X^t)^\ast v\|_2 = \bigg(\sum^n_{i=1} \|(\partial_i P)(X)^\ast v\|_2^2\bigg)^{1/2} =  \|(\partial P)(X)^\ast v\|_2.$$

(ii) We note that for each $Q\in\C\langle Z_1,\dots,Z_n\rangle^{\otimes 2}$
$$\limsup_{t\searrow 0} \|Q(X^t)\|_{\pi,X^t} \leq \|Q(X)\|_{\pi,X}.$$
Indeed, for any given $Q\in\C\langle Z_1,\dots,Z_n\rangle^{\otimes 2}$, we may consider an arbitrary decomposition
$$Q = \sum^N_{k=1} Q_{k,1} \otimes Q_{k,2}.$$
For $k=1,\dots,N$, we may write
$$Q_{k,1}(X^t) \otimes Q_{k,2}(X^t) = Q_{k,1}(X) \otimes Q_{k,2}(X) + \sum^{d_k}_{l=1} t^{l/2} R_{k,l}$$
for some $d_k\geq1$ with certain elements
$$R_{k,1},\dots,R_{k,{d_k}}\in \C\langle X_1,\dots,X_n,S_1,\dots,S_n\rangle,$$
which are independent of $t$. Since the norm $\|\cdot\|$ on $M\overline{\otimes} M$ is a cross norm, we get
\begin{align*}
\|Q_{k,1}(X^t)\| \|Q_{k,2}(X^t)\| &= \|Q_{k,1}(X^t) \otimes Q_{k,2}(X^t)\|\\
                                  &\leq \|Q_{k,1}(X) \otimes Q_{k,2}(X)\| + \sum^{d_k}_{l=1} t^{l/2} \|R_{k,l}\|\\
                                  &= \|Q_{k,1}(X)\| \|Q_{k,2}(X)\| + \sum^{d_k}_{l=1} t^{l/2} \|R_{k,l}\|
\end{align*}
for all $k=1,\dots,N$ and thus
\begin{align*}
\|Q(X^t)\|_{\pi,X^t} &\leq \sum^N_{k=1} \|Q_{k,1}(X^t)\| \|Q_{k,2}(X^t)\|\\
                     &\leq \sum^N_{k=1} \|Q_{k,1}(X)\| \|Q_{k,2}(X)\| + \sum^N_{k=1} \sum^{d_k}_{l=1} t^{l/2} \|R_{k,l}\|,
\end{align*}
so that
$$\limsup_{t\searrow 0} \|Q(X^t)\|_{\pi,X^t} \leq \sum^N_{k=1} \|Q_{k,1}(X)\| \|Q_{k,2}(X)\|.$$
Since the decomposition of $Q$ was arbitrarily chosen, we get as desired
$$\limsup_{t\searrow 0} \|Q(X^t)\|_{\pi,X^t} \leq \|Q(X)\|_{\pi,X}.$$

(iii) Let $Q\in\C\langle Z_1,\dots,Z_n\rangle^{\otimes 2}$ be given. Since $\Phi^\ast(X^t)<\infty$, we obtain by Proposition \ref{Fisher-bound} that 
\begin{align*}
\lefteqn{\sum^n_{i=1} |\langle v_t^\ast (\partial_i P)(X^t) u_t, Q(X^t)\rangle|^2}\\
 & \quad \leq 16 \big(\|P(X^t) u_t\|_2 \|v_t\| + \|u_t\| \|P(X^t)^\ast v_t\|_2\big)^2 \Phi^\ast(X^t) \|Q(X^t)\|_{\pi,X^t}^2\\
 & \quad \leq 16 \big(\|P(X^t) u\|_2 \|v\| + \|u\| \|P(X^t)^\ast v\|_2\big)^2 \Phi^\ast(X^t) \|Q(X^t)\|_{\pi,X^t}^2\\
 & \quad = 16 \bigg(\frac{1}{\sqrt{t}} \|P(X^t) u\|_2 \|v\| + \|u\| \frac{1}{\sqrt{t}} \|P(X^t)^\ast v\|_2 \bigg)^2 t \Phi^\ast(X^t)  \|Q(X^t)\|_{\pi,X^t}^2,
\end{align*}
so that, since $\displaystyle{\limsup_{t\searrow 0} \|Q(X^t)\|_{\pi,X^t} \leq \|Q(X)\|_{\pi,X}}$ according to (ii),
\begin{align*}
\lefteqn{\liminf_{t\searrow 0}\ \sum^n_{i=1} |\langle v_t^\ast (\partial_i P)(X^t) u_t, Q(X^t)\rangle|^2}\\
 & \qquad \leq 16 \big(\|(\partial P)(X)u\|_2 \|v\| + \|u\| \|(\partial P)(X)^\ast v\|_2\big)^2 \alpha(X) \|Q(X)\|_\pi^2.
\end{align*}

(iv) It remains to show that in fact
$$\liminf_{t\searrow 0}\ \sum^n_{i=1} |\langle v_t^\ast (\partial_i P)(X^t) u_t, Q(X^t)\rangle|^2 = \sum^n_{i=1} |\langle v^\ast (\partial_i P)(X) u, Q(X)\rangle|^2.$$
We first check that
\begin{align*}
\langle v_t^\ast (\partial_i P)(X^t) u_t, Q(X^t)\rangle
&= \langle \E_t[v^\ast] (\partial_i P)(X^t) \E_t[u], Q(X^t)\rangle\\
&= \langle (\E_t \otimes \E_t)[v^\ast (\partial_i P)(X^t) u], Q(X^t)\rangle\\
&= \langle v^\ast (\partial_i P)(X^t) u, Q(X^t)\rangle,
\end{align*}
which means in particular that this expression is actually a complex polynomial in $\sqrt{t}$. This guarantees that
$$\lim_{t\searrow0}\ \langle v_t^\ast (\partial_i P)(X^t) u_t, Q(X^t)\rangle =\langle v^\ast (\partial_i P)(X) u, Q(X)\rangle$$
and finally completes the proof.
\end{proof}

\subsection{Proof of Theorem \ref{main2}}
\label{subsec:proof_thm}

Now, we are prepared to give a proof of Theorem \ref{main2}. Even more, we will do this under the (possibly) weaker assumption that $\hat\delta^\star(X_1,\dots,X_n)=n$.

The main tool is the following corollary of Proposition \ref{Fisher-growth}.

\begin{corollary}\label{reduction-strong}
Let $(M,\tau)$ be a tracial $W^\ast$-probability space and let self-adjoint elements $X_1,\dots,X_n \in M$ be given, such that
$$\hat\delta^\star(X_1,\dots,X_n) = n$$
holds. We consider $P\in\C\langle Z_1,\dots,Z_n\rangle$. Then, for arbitrary $u,v\in\vN(X_1,\dots,X_n)$, the following implication holds true:
$$P(X) u = 0 \quad\text{and}\quad P(X)^\ast v = 0 \implies \forall i=1,\dots,n:\ v^\ast (\partial_i P)(X) u  = 0.$$
\end{corollary}

\begin{proof}
Note that our assumption $\hat\delta^\star(X_1,\dots,X_n) = n$ is equivalent to the fact that $\alpha(X_1,\dots,X_n)=0$.
If $P(X) u = 0$ and $P(X)^\ast v = 0$, then Proposition \ref{Fisher-growth} yields
$$\sum^n_{i=1} |\langle v^\ast (\partial_i P)(X) u, Q\rangle|^2 = 0$$
for all $Q\in\C\langle X_1,\dots,X_n\rangle^{\otimes 2}$. Since $\C\langle X_1,\dots,X_n\rangle^{\otimes 2}$ is by definition dense in $L^2(M,\tau)$ with respect to $\|\cdot\|_2$, we conclude that  $v^\ast (\partial_i P)(X) u = 0$ for $i=1,\dots,n$.
\end{proof}

\begin{remark}
Let $P\in\C\langle X_1,\dots,X_n\rangle$ be given and assume that there are $u=u^\ast,v=v^\ast \in\vN(X_1,\dots,X_n)$ such that $Pu=P^\ast v=0$ holds. Then, according to Corollary \ref{reduction-strong}, we know that $v\otimes 1 (\partial_j P) 1 \otimes u = 0$ for any $j=1,\dots,n$. If we replace now $P$ by $P^\ast$, the statement of Corollary \ref{reduction-strong} also gives $u\otimes 1(\partial_jP^\ast)1\otimes v=0$ for $j=1,\dots,n$. But we want to point out that this does not lead to any new information. Indeed, if we take adjoints in the initial statement
$$v\otimes 1(\partial_j P)1\otimes u = 0,$$
we get
$$1\otimes u(\partial_j P)^\ast v \otimes 1 = 0.$$
Then, if we apply the flip $\sigma: M \otimes M \to M \otimes M$, i.e. the $\ast$-homomorphism induced by $\sigma(a_1 \otimes a_2) := a_2 \otimes a_1$, it follows that
$$u\otimes 1\sigma((\partial_j P)^\ast)1\otimes v = 0.$$
An easy calculation on monomials shows $\sigma((\partial_j P)^\ast) = \partial_j P^\ast$, such that the above result reduces exactly to the statement obtained by replacing $P$ with $P^\ast$.
\end{remark}

Before doing the final step, we first want to test in two examples how strong the result in Corollary \ref{reduction-strong} is.

\begin{example}
For the self-adjoint polynomial $P=X_1X_2X_1$, we calculate $\partial_2 P = X_1 \otimes X_1$, such that $Pw=0$ implies according to Corollary \ref{reduction-strong} that $wX_1 \otimes X_1w = 0$ and therefore $X_1w=0$ holds.

Applying now Corollary \ref{reduction-strong} once again with respect to $\partial_1$, we end up with $w\otimes w=0$, such that $w=0$ follows.
\end{example}

\begin{example}
Take $P=X_1X_2+X_2X_1$. We have
\begin{align*}
\partial_1 P &= 1 \otimes X_2 + X_2 \otimes 1,\\
\partial_2 P &= X_1 \otimes 1 + 1 \otimes X_1
\end{align*}
and thus according to Corollary \ref{reduction-strong}
\begin{align*}
(X_2 w)^\ast (X_2 w) &= m_{X_2}( w \otimes 1 (\partial_1 P) 1 \otimes w) = 0,\\
(X_1 w)^\ast (X_1 w) &= m_{X_1}( w \otimes 1 (\partial_2 P) 1 \otimes w) = 0.
\end{align*}
We conclude $X_1w=X_2w=0$, from which we may deduce like above by a second application of Corollary \ref{reduction-strong} that $w=0$.
\end{example}

Although the above examples might give the feeling that Corollary \ref{reduction-strong} is strong enough to allow directly a successive reduction of any polynomial, the needed algebraic manipulations turn out to be obscure in general; a skeptical reader might convince himself by having a try at the polynomial $P=X_1X_2X_3+X_3X_2X_1$, for instance. 

Moreover, in contrast to Theorem \ref{main2}, any such reduction argument that is based only on Corollary \ref{reduction-strong} would need a symmetric starting condition. Fortunately, in our situation, we can go around these complications, since we can use the following well-known general lemma, which is an easy consequence of the polar decomposition and encodes the additional information that our statement is formulated in a tracial setting. 

\begin{lemma}\label{key-lemma}
Let $x$ be an element of any tracial $W^\ast$-probability space $(M,\tau)$ over some complex Hilbert space $H$. Let $p_{\ker(x)}$ and $p_{\ker(x^\ast)}$ denote the orthogonal projections onto $\ker(x)$ and $\ker(x^\ast)$, respectively.

The projections $p_{\ker(x)}$ and $p_{\ker(x^\ast)}$ belong both to $M$ and satisfy
$$\tau(p_{\ker(x)}) = \tau(p_{\ker(x^\ast)}).$$
Thus, in particular, if $\ker(x)$ is non-zero, then also $\ker(x^\ast)$ is a non-zero subspace of $H$.
\end{lemma}

\begin{proof}
We consider the polar decomposition $x = v(x^\ast x)^{1/2} = (x x^\ast)^{1/2} v$ of $x$, where $v\in M$ is a partial isometry mapping $\overline{\ran(x^\ast)}$ to $\overline{\ran(x)}$, such that
$$v^\ast v = p_{\overline{\ran(x^\ast)}} \qquad\text{and}\qquad v v^\ast = p_{\overline{\ran(x)}}.$$
Hence, it follows that
$$1-v^\ast v = p_{\ran(x^\ast)^\bot} = p_{\ker(x)} \qquad\text{and}\qquad 1 - v v^\ast = p_{\ran(x)^\bot} = p_{\ker(x^\ast)},$$
from which we may deduce by traciality of $\tau$ that indeed
$$\tau(p_{\ker(x)}) = \tau(1-v^\ast v) = \tau(1 - v v^\ast) = \tau(p_{\ker(x^\ast)}).$$
This concludes the proof.
\end{proof}

Combining Lemma \ref{key-lemma} with Corollary \ref{reduction-strong} will provide us with the desired reduction argument. Before giving the precise statement, let us introduce some notation. If $p\in M$ is any projection, we define a linear mapping
$$\Delta_{p,j}:\ \C\langle Z_1,\dots,Z_n\rangle \rightarrow \C\langle Z_1,\dots,Z_n\rangle$$
for $j=1,\dots,n$ by 
$$\Delta_{p,j} P := (\tau \otimes \id)((p\otimes 1) (\ev_X \otimes \id)(\partial_j P))$$ 
for any $P\in \C\langle Z_1,\dots,Z_n\rangle$.

\begin{corollary}\label{reduction-final}
Let $P\in \C\langle Z_1,\dots,Z_n\rangle$ and $w=w^\ast\in\vN(X_1,\dots,X_n)$ be given, such that $P(X)w=0$ holds true. If $w\neq0$, then there exists a projection $0\neq p\in\vN(X_1,\dots,X_n)$ such that $(\Delta_{p,j} P)(X) w = 0$.
\end{corollary}

\begin{proof}
Since $P(X)w=0$ and $w\neq0$, we see that $\{0\} \neq \ran(w) \subseteq \ker(P(X))$, such that we also must have $\ker(P(X)^\ast) \neq \{0\}$ according to Lemma \ref{key-lemma}. The projection $p := p_{\ker(P(X)^\ast)} \in\vN(X_1,\dots,X_n)$ thus satisfies $p\neq0$ and $P(X)^\ast p = 0$. Corollary \ref{reduction-strong} tells us that $(p\otimes 1) (\partial_j P)(X) (1 \otimes w) = 0$ for $j=1,\dots,n$ holds true. Hence, we get that
\begin{align*}
(\Delta_{p,j} P)(X) w &= (\tau\otimes\id)((p\otimes 1) (\partial_j P)(X)) w\\
                      &= (\tau\otimes\id)((p\otimes 1) (\partial_j P)(X) (1 \otimes w))\\
                      &= 0,
\end{align*}
which concludes the proof.
\end{proof}

Now, we are prepared to finish the proof of Theorem \ref{main2}.

\begin{proof}[Proof of Theorem \ref{main2}]
Obviously, it suffices to show that, if any non-commutative polynomial $P\in\C\langle Z_1,\dots,Z_n\rangle$ and $w\in\vN(X_1,\dots,X_n)$ with $w\neq0$ are given such that $P(X)w=0$, then $P=0$ follows. By possibly replacing $w$ by $ww^\ast$, we may assume in addition that $w=w^\ast$.

For proving $P=0$, we proceed as follows. First, we write
$$P = a_0 + \sum_{k=1}^d \sum_{i_1,\dots,i_k = 1}^n a_{i_1,\dots,i_k} Z_{i_1} \dots Z_{i_k}$$
and we assume that the total degree $d$ of $P$ satisfies $d\geq1$. We choose any summand of highest degree
$$a_{i_1,\dots,i_d} Z_{i_1} \dots Z_{i_d}$$
of $P$ which is non-zero. Iterating Corollary \ref{reduction-final}, we see that there are non-zero projections $p_1,\dots,p_d \in\vN(X_1,\dots,X_n)$ such that
$$(\Delta_{p_d,i_d} \dots \Delta_{p_1,i_1} P)(X) w = 0.$$
But since we can easily check that
$$(\Delta_{p_d,i_d} \dots \Delta_{p_1,i_1} P)(X) = \tau(p_d) \dots \tau(p_1) a_{i_1,\dots,i_d},$$
this leads us to $a_{i_1,\dots,i_d} = 0$, which contradicts our assumption. Thus, $P$ must be constant, and since $w\neq0$, we end up with $P=0$. This concludes the proof of Theorem \ref{main2}.
\end{proof}

We finish by noting that Theorem \ref{main2} yields now the following generalization of Theorem \ref{main1}.

\begin{corollary}
Let $(M,\tau)$ be a tracial $W^\ast$-probability space. Furthermore, let $X_1,\dots,X_n \in M$ be self-adjoint elements and assume that $\delta^\ast(X_1,\dots,X_n) = n$ holds. Then the following statements hold true:
\begin{itemize}
 \item[(a)] $X_1,\dots,X_n$ do not satisfy any non-trivial algebraic relation, i.e. there exists no non-zero polynomial $P\in\C\langle Z_1,\dots,Z_n\rangle$ such that $P(X_1,\dots,X_n)=0$.
 \item[(b)] For $j=1,\dots,n$, there is a unique derivation
$$\hat{\partial}_j:\ \C\langle X_1,\dots,X_n\rangle \rightarrow \C\langle X_1,\dots,X_n\rangle \otimes \C\langle X_1,\dots,X_n\rangle$$
which satisfies $\hat{\partial}_j(X_i) = \delta_{j,i} 1 \otimes 1$ for $i=1,\dots,n$.
\end{itemize}
\end{corollary}

\begin{appendix}

\section{Non-commutative derivatives and algebraic relations}

In order to exclude non-trivial algebraic relations among a collection of self-adjoint variables $X=(X_1,\dots,X_n)$ in some $W^\ast$-probability space $(M,\tau)$, satisfying the condition $\Phi^\ast(X_1,\dots,X_n) < \infty$, we used in Section \ref{algebraic_relations} (to be more precise, in the proof of Proposition \ref{key-prop-1}) a certain reduction argument. This was based on \eqref{reduction}, namely the implication
$$P(X_1,\dots,X_n) = 0 \quad\Longrightarrow\quad \forall j=1,\dots,n:\ (\partial_j P)(X_1,\dots,X_n) = 0.$$
But in fact, this property is furthermore equivalent to the existence of certain non-commutative derivatives on $\C\langle X_1,\dots,X_n\rangle$. These relations are clarified by the following Proposition.

\begin{proposition}\label{derivatives_relations}
Let $(M,\tau)$ be a tracial $W^\ast$-probability space and let self-adjoint $X_1,\dots,X_n\in M$ be given. Then the following statements are equivalent:
\begin{itemize}
 \item[(i)] The variables $X_1,\dots,X_n$ do not satisfy any algebraic relation.
 \item[(ii)] The following implication holds true for any non-commutative polynomial $P\in \C\langle Z_1,\dots,Z_n\rangle$:
 $$P(X_1,\dots,X_n) = 0 \quad\Longrightarrow\quad \forall j=1,\dots,n:\ (\partial_j P)(X_1,\dots,X_n) = 0$$
 \item[(iii)] For each $j=1,\dots,n$, there is a derivation
$$\hat{\partial}_j:\ \C\langle X_1,\dots,X_n\rangle \rightarrow \C\langle X_1,\dots,X_n\rangle \otimes \C\langle X_1,\dots,X_n\rangle$$
such that the following diagram commutes.
\begin{equation}\label{derivation_diagram}
\begin{xy}
 \xymatrix{\C\langle Z_1,\dots, Z_n\rangle \ar[rr]^{\partial_j} \ar@{->>}[dd]_{\ev_X} & & \C\langle Z_1,\dots,Z_n\rangle^{\otimes 2}\ar@{->>}[dd]^{\ev_X \otimes \ev_X}\\
          & & \\
           \C\langle X_1,\dots,X_n\rangle \ar[rr]^{\partial_j} \ar[rr]^{\hat{\partial}_j} & & \C\langle X_1,\dots,X_n\rangle^{\otimes 2}}
\end{xy}
\end{equation}
 \item[(iv)] For each $j=1,\dots,n$, there is a derivation
$$\hat{\partial}_j:\ \C\langle X_1,\dots,X_n\rangle \rightarrow \C\langle X_1,\dots,X_n\rangle \otimes \C\langle X_1,\dots,X_n\rangle$$
such that $\hat{\partial}_j(X_i) = \delta_{j,i} 1 \otimes 1$ for $i=1,\dots,n$.
 \end{itemize}
In particular, if the equivalent conditions (i) -- (iv) are satisfied, then each derivation $\hat{\partial}_j$ in (iii) as well as in (iv) is uniquely determined, and they both coincide.
\end{proposition}

\begin{proof}
The implication ``(i) $\Longrightarrow$ (ii)'' was already shown in Proposition \ref{key-prop-2} and the converse implication ``(ii) $\Longrightarrow$ (i)'' is trivial.

\medskip

\underline{(ii) $\Longrightarrow$ (iii):} For any element $Y\in \C\langle X_1,\dots,X_n\rangle$, we choose any $P\in\C\langle X_1,\dots,X_n\rangle$ with $Y=P(X_1,\dots,X_n)$ and we put
$$\hat{\partial}_j Y := (\partial_j P)(X_1,\dots,X_n) \qquad\text{for $j=1,\dots,n$}.$$
By assumption (i), this gives a well-defined mapping
$$\hat{\partial}_j:\ \C\langle X_1,\dots,X_n\rangle \rightarrow \C\langle X_1,\dots,X_n\rangle \otimes \C\langle X_1,\dots,X_n\rangle,$$
which is in fact a derivation, as a straightforward calculation shows. Moreover, by definition of $\hat{\partial}_j$, it is clear that the diagram in \ref{derivation_diagram} commutes.

\medskip

\underline{(iii) $\Longrightarrow$ (ii):} Let $P\in\C\langle Z_1,\dots,Z_n\rangle$ with $P(X_1,\dots,X_n)=0$ be given. By assumption, we get
\begin{align*}
(\partial_j P)(X_1,\dots,X_n) &= ((\ev_X \otimes \ev_X) \circ \partial_j)(P)\\
                              &= (\hat{\partial}_j \circ \ev_X)(P)\\
															&= \hat{\partial}_j(P(X_1,\dots,X_n))\\
															&= \hat{\partial}_j(0)\\
															&= 0,
\end{align*}
which shows (ii).

\medskip

\underline{(iii) $\Longrightarrow$ (iv):} Let any $j=1,\dots,n$ be given. If there is a derivation $\hat{\partial}_j$, for which the diagram in \eqref{derivation_diagram} commutes, then we can check that
\begin{align*}
\hat{\partial}_j(X_i) &= \hat{\partial}_j(\ev_X(Z_i))\\
                      &= (\ev_X \otimes \ev_X)(\partial_j Z_i)\\
                      &= (\ev_X \otimes \ev_X)(\delta_{j,i} 1 \otimes 1)\\
                      &= \delta_{j,i} 1 \otimes 1.
\end{align*}
holds for $i=1,\dots,n$, i.e. $\hat{\partial}_j$ satisfies the condition of (iii).

\medskip

\underline{(iv) $\Longrightarrow$ (iii):} If there exits, for any $j=1,\dots,n$, a derivation
$$\hat{\partial}_j:\ \C\langle X_1,\dots,X_n\rangle \rightarrow \C\langle X_1,\dots,X_n\rangle \otimes \C\langle X_1,\dots,X_n\rangle$$
such that $\hat{\partial}_j(X_i) = \delta_{j,i} 1 \otimes 1$ holds for $i=1,\dots,n$, then
$$d_j := \hat{\partial}_j \circ \ev_X:\ \C\langle Z_1,\dots, Z_n\rangle \rightarrow \C\langle X_1,\dots,X_n\rangle^{\otimes 2}.$$
obviously defines a derivation, where we consider $\C\langle X_1,\dots,X_n\rangle^{\otimes 2}$ as $\C\langle Z_1,\dots,Z_n\rangle$-bimodule via evaluation $\ev_X$. Thus, we have for each $P\in\C\langle Z_1,\dots,Z_n\rangle$ according to \eqref{derivation-uniqueness} that
$$d_j(P) = \sum_{i=1}^n (\ev_X\otimes\ev_X)(\partial_i P) \sharp d_j(Z_i).$$
Since $d_j(Z_i) = \hat{\partial}_j(X_i) = \delta_{j,i} 1 \otimes 1$ holds for $i=1,\dots,n$ by assumption, we see that $d_j(P) = (\ev_X\otimes\ev_X)(\partial_j P)$. By definition of $d_j$, this gives
$$(\hat{\partial}_j\circ\ev_X)(P) = d_j(P) = ((\ev_X\otimes\ev_X) \circ \partial_j) (P),$$
which precisely means that the diagram in \eqref{derivation_diagram} commutes.

\medskip

If the equivalent statements (i) -- (iv) are all valid, then the derivations in (iii) and (iv) are uniquely determined. Indeed, if $\hat{\partial}_j$ satisfies (iii) or (iv), then its value on each element in $\C\langle X_1,\dots,X_n\rangle$, represented according to (i) as $P(X_1,\dots,X_n)$ for some unique $P\in\C\langle Z_1,\dots,Z_n\rangle$, must be given by
$$\hat{\partial}_j (P(X_1,\dots,X_n)) = (\hat{\partial}_j \circ \ev_X)(P) = (\ev_X \circ \partial_j)(P) = (\partial_j P)(X_1,\dots,X_n)$$
(according to the commutativity of the diagram in \eqref{derivation_diagram}) or
$$\hat{\partial}_j (P(X_1,\dots,X_n)) = (\hat{\partial}_j\circ\ev_X)(P) = d_j(P) = (\partial_j P)(X_1,\dots,X_n)$$
(according to the proof of ``(iv) $\Longrightarrow$ (iii)''), respectively. Furthermore, we see that both derivations must coincide.
\end{proof}

\end{appendix}

\bibliographystyle{amsalpha}
\bibliography{zero_divisors}

\end{document}